\documentclass[11pt,reqno]{amsart}

\usepackage{amsmath, amsfonts, amsthm, amssymb}

\newtheorem{Theorem}{Theorem}[section]
\newtheorem{Lemma}{Lemma}[section]
\newtheorem{Proposition}{Proposition}[section]

\theoremstyle{definition}

\theoremstyle{remark}
\newtheorem{Remark}{Remark}[section]

\usepackage{amsmath}

\renewcommand{\u}{{\bf u}}

\newcommand{\R}{{\mathbb R}}
\newcommand{\Dv}{{\rm div}}

\def\f{\frac}
\renewcommand{\O}{\Omega}

\def\D{\Delta }

\def\hf1{^\f{1}{1-\xi^2}}

\def\be{\begin{equation}}
\def\en{\end{equation}}
\def\bs{\begin{split}}
\def\es{\end{split}}

\author{Cheng Yu}
\address{Department of Mathematics,  The University of Texas,
                           Austin, Texas 78712.}
\email{yucheng@math.utexas.edu}

\title
[]
{ The energy equality for the Navier-Stokes equations in bounded domains}

\subjclass[2000]{}

\date{\today}

\begin{document}
\begin{abstract}
In this paper, we provide a sufficient condition of the energy equality for the incompressible Navier-Stokes equations in bounded domains.
\end{abstract}

\maketitle
\section{Introduction}
This paper is concerned with the question of the energy conservation for the weak solutions of Navier-Stokes equations
\begin{equation}
\label{NS equations}
\begin{split}
& \u_t+\u\cdot\nabla\u+\nabla P-\mu\D\u=0,
\\&\Dv\u=0,
\end{split}
\end{equation}
with the initial data
\begin{equation}
\label{initial data}
\u(0,x)=\u_0
\end{equation}
for $(t,x)\in \R^{+}\times\O,$ and boundary condition \begin{equation}
\label{boundary condition}
\u=0\;\;\text{ on }\partial\Omega,
\end{equation} where $\O$ is a bounded domain in $\R^d$, $d=2$ or $3.$
\vskip0.3cm

The concept of weak solution has been introduced  in \cite{Le, H}. As we all know it, a weak solution $\u$ satisfies the energy inequality
\begin{equation}
\label{energy inequality}
\int_{\O}|\u(t,x)|^2\,dx+2\mu\int_0^t\int_{\O}|\nabla\u|^2\,dx\,dt\leq \int_{\O}|\u_0|^2\,dx,
\end{equation}
for any $t\in[0,T].$
Given that the solution to the Navier-Stokes is sufficiently smooth, the energy equality could hold for any time. However, the existence of smooth solutions is a longstanding open problem in three dimensional space. Thus, an interesting question to ask is how badly behavior of a weak solution  $\u$ can keep the energy conservation.
Mathematically, what is minimal regularity such that a weak solution satisfies the following energy equality,
\begin{equation}
\label{energy conservation}
\int_{\O}|\u(t,x)|^2\,dx+2\mu\int_0^t\int_{\O}|\nabla\u|^2\,dx\,dt=\int_{\O}|\u_0|^2\,dx?
\end{equation}
This problem has been studied by
Serrin \cite{Serrin} and Shinbrot \cite{Shinbrot}
when $\O$ is all of $\R^d$ or the periodic domain. In particular, Serrin has proved that $\u$ is smooth and satisfies \eqref{energy conservation} if $\u \in L^p(0,T;L^q(\mathbb{T}^d))$, where
\begin{equation}
\label{Serrion condition}
\frac{2}{p}+\frac{d}{q}\leq 1.
\end{equation}
 Shinbrot \cite{Shinbrot} has shown the energy equality \eqref{energy conservation} holds if $\u \in L^p(0,T;L^q(\mathbb{T}^3))$, where
\begin{equation}
\label{Shinbrot condition}
\frac{1}{p}+\frac{1}{q}\leq \frac{1}{2},\; q\geq 4.
\end{equation}
\vskip0.3cm
 On the Euler equations, it is linked to the name of the ``Onsager conjecture" \cite{O}: non-conservation of energy in the three-dimensional Euler equations would be related to the loss of regularity. Specifically, Onsager conjectured that every weak solution to the Euler equations with H\"older continuity exponent $\alpha > {1\over 3}$ conserves energy; and anomalous dissipation of energy occurs when $\alpha < {1\over 3}$. See \cite{DS13,R03} for reviews and further discussions.
The first part of the conjecture was proved by Eyink \cite{Ey}, Constantin-E-Titi \cite{CET},  among others.
The development toward the other direction of the conjecture is more recent, see  \cite{DS13b,DS14, I,I2}.
 Very recently,
 Bardos-Titi \cite{BT} considered the boundary effects and extended the classical result of  Constantin-E-Titi \cite{CET} to the case of the bounded domain.  It would be an interesting work to extend the work of \cite{BT} to the Navier-Stokes equations.

\vskip0.3cm

We are particularly interested in
 investigating the relation between the energy equality and the degree of regularity of the solutions for system \eqref{NS equations}-\eqref{boundary condition}. In particular we provide a sufficient conditions on the regularity of solutions to ensure energy equality. Our approach is in the spirit of Bardos-Titi \cite{BT}, Constantin-E-Titi \cite{CET} and Kato \cite{K}.
 The main difficulties, which also constitutes the main contribution of this paper, are explained as follows.

\subsubsection*{Less regularity of $\u$}
In contrast to the Onsager conjecture of the Euler equations, the energy equality of the Navier-Stokes equations has less regularity assumption. In fact, our goal is to look for the condition which requires the regularity below the ones of \cite{Serrin}. In \cite{BT}, the bound of $\u$ in $L^{\infty}$ could be obtained and it is crucial to control the boundary effects. However, this bound is not allowed to use in this paper because it is smooth enough for the solution of Navier-Stokes equations.  Thus,
we have to take a careful analysis for each term and boundary effects under different regularity space. However, a benefit of the diffusion term is that  $\nabla\u$ is bounded in $L^2(0,T;L^2(\O))$.
\subsubsection*{Appearance of the boundary layer}
 We have to pay attention on the boundary effects of the diffusion term for the Navier-Stokes equations, which is different from the Euler equations.
  The key step is to control the integration on the diffusion near boundary, which is similar to
   the famous open problem which was formulated by Kato \cite{K}. In particular, it is not known if "$L^2-\text{strength}$" of this layer goes to zero, and more precisely, the limit of solutions of the Navier-Stokes equations with Dirichlet boundary conditions to the solutions of Euler equations.

\vskip0.3cm
Our current work is motivated by the recent work of  Bardos-Titi \cite{BT} where they used the cut-off argument to  prove the Onsager conjecture in bounded domains. To state our result and proof,  we adopt their notation with respect to the boundary:
\\for any $x\in \overline{\partial\O},$ $d(x)=\inf_{y\in \partial\O}|x-y|$, and the open set $\O_h=\{x\in \O|d(x)<h\}.$\\
From now, we assume that $\partial\O$ is $C^2$ compact manifold, thus there exists $h_0(\O)>0$ with the following properties:
\\1. For any $x\in \overline{\O_{h_0}}$, $d(x)$ belongs to $C^1(\overline{\O_{h_{0}}})$;
\\2. For any $x\in \overline{\O_{h_0}}$, there exists a unique point $\delta(x)\in \partial \O$ such that
$$d(x)=|x-\delta(x)|\quad\quad\text{ and one has }\;\;\nabla d(x)=-\vec{n}(\delta(x)).$$
Next, we defined $$\theta_{h}(x)=\kappa(\frac{d(x,\partial\O)}{h}),$$
where
$$s\mapsto\kappa(s)\in C^{\infty}(\O),\;\kappa'(s)\geq 0,$$
and $\kappa(s)=0$ for $s\leq 1$, and $\kappa(s)=1$ for $s\geq 2$. In addition, $$|\kappa'(t)-\kappa'(s)|\leq L|t-s|,$$
 where $L$ is a Lipschitz constant. Note that if $h\leq h_0(\O)$, then $$d(x,\partial\O)=|x-\sigma(x)|\in C^2(\O).$$
The following is our main result of this paper.

\begin{Theorem}
\label{main result}
Let $\u\in L^{\infty}(0,T;L^2(\O))\cap L^2(0,T;H^1(\O))$ be a weak solution of the incompressible Navier-Stokes equations, that is,
\begin{equation}
\label{weak formulation}
\begin{split}
-\int_0^T\int_{\O}\u\varphi_t\,dx\,dt&-\int_{\O}\u_0\varphi(0,x)\,dx-\int_0^T\int_{\O}\nabla\varphi\u\otimes\u\,dx\,dt
\\&+\mu\int_0^T\int_{\O}\nabla\u\nabla\varphi\,dx\,dt=0
\end{split}
\end{equation}
for any smooth test function $\varphi\in C^{\infty}(\R^{+}\times\O)$ with compact support, and $\Dv\varphi=0.$
In addition, if
$$\u \in L^p(0,T;L^q(\O))\cap L^s(0,T;B_s^{\alpha,\infty}(\Omega)),$$
 for any $\frac{1}{p}+\frac{1}{q}\leq  \frac{1}{2}, \; q\geq 4$, $s>2$ and for any $\frac{1}{2}+\frac{1}{s}<\alpha< 1$, then
\begin{equation*}
\int_{\O}|\u(t,x)|^2\,dx+2\mu\int_0^t\int_{\O}|\nabla\u|^2\,dx\,dt=\int_{\O}|\u_0|^2\,dx
\end{equation*}
for any $t\in [0,T].$
\end{Theorem}

\begin{Remark} Compared to the work of \cite{Shinbrot} in a periodic domain, we need additional condition on $\u$
$$\u\in L^s(0,T;B_s^{\alpha,\infty}(\O)).$$ This allows us to treat with the boundary effects with respect to the diffusion term.
\end{Remark}
\begin{Remark} It is interesting to extend our previous results \cite{CY,Yu} on the compressible flows in the setting of the bounded domains.
\end{Remark}
\begin{Remark} Since $\u\in L^{\infty}(0,T;L^2(\O))\cap L^p(0,T;L^q(\O))$ for any $\frac{1}{p}+\frac{1}{q}\leq \frac{1}{2}, \; q\geq 4$, one can deduce that
$$\|\u\|_{L^4(0,T;L^4(\O))}\leq C\|\u\|_{L^{\infty}(0,T;L^2(\O))}^{a}\|\u\|^{1-a}_{L^p(0,T;L^q(\O))}$$ for
some $0<a<1$. Thus, we can use a fact that $\u$ is bounded in $L^4(0,T;L^4(\O))$ in our proof. On the other hand, this estimate yields the uniqueness of weak solutions in the Leray sense.
\end{Remark}

\vskip0.3cm
\section{Proof}

The goal of this section is to prove our main result. Our proof relies on the following lemma, see \cite{L}.
 \begin{Lemma}
 \label{Lions's lemma}
 Let $f\in W^{1,p}(\R^d),\,g\in L^{q}(\R^d)$ with $1\leq p,q\leq \infty$, and $\frac{1}{p}+\frac{1}{q}\leq 1$. Then, we have
 $$\|\Dv(fg)*\eta_{\varepsilon}-\Dv(f(g*\eta_{\varepsilon}))\|_{L^{r}(\R^d)}\leq C\|f\|_{W^{1,p}(\R^d)}\|g\|_{L^{q}(\R^d)}$$
 for some $C\geq 0$ independent of $\varepsilon$, $f$ and $g$, $r$ is determined by $\frac{1}{r}=\frac{1}{p}+\frac{1}{q}.$ In addition,
  $$\Dv(fg)*\eta_{\varepsilon}-\Dv(f(g*\eta_{\varepsilon}))\to0\;\;\text{ in }\,L^{r}(\R^d)$$
 as $\varepsilon \to 0$ if $r<\infty.$ Here  $\varepsilon>0$ is a small enough number, $\eta\in C_0^{\infty}(\O)$ be a standard mollifier supported in $B(0,1).$
 \end{Lemma}
\vskip0.3cm

Here we will rely on the following lemma, which was proved in \cite{CY}.
 \begin{Lemma}
 \label{key lemma}
 Let $f \in B_p^{\alpha,\infty}( \O)$, $g\in L^q (\O)$ with $1\leq p,\,q \leq \infty.$  Then, we have
$$\|(fg)^{\varepsilon}-f\,g^{\varepsilon}\|_{L^r(\O)}\leq C\| f\|_{B_p^{\alpha,\infty}( \O)}\|g\|_{L^q(\O)}$$
for some constant $C>0$ independent of  $f$ and $g$, and with $\frac{1}{r}=\frac{1}{p}+\frac{1}{q}$. In addition,
$$\|(fg)^{\varepsilon}-f\,g^{\varepsilon}\|_{ L^r(\O)}\leq C\varepsilon^{\alpha}\to 0$$
as $\varepsilon\to 0$ if $r<\infty.$
 \end{Lemma}
\begin{Remark} \label{remark for proof}
If $f$ is a Lipschitz function $|f(t,x)-f(t,y)|\leq L|x-y|,$ which means $\alpha=1$, conclusion holds.
\end{Remark}
We  introduce our test function $\Phi^{h,\varepsilon}=\theta_h((\theta_h\u)^{\varepsilon})^{\varepsilon}$, where $f^{\varepsilon}=f*\eta_{\varepsilon}(x)$, $\varepsilon>0$ is a small enough number, $\eta\in C_0^{\infty}(\O)$ be a standard mollifier supported in $B(0,1).$
Using $\Phi^{h,\varepsilon}$ to test Navier-Stokes equations \eqref{NS equations}, one obtains
\begin{equation}
\label{first step equality}
\int_{\O}\Phi^{h,\varepsilon}\left(\u_t+\Dv(\u\otimes\u)+\nabla P-\mu\Delta\u\right)\,dx=0.
\end{equation}
To prove our main result, we investigate \eqref{first step equality} term by term in the following steps:
\vskip0.3cm
\textbf{Step 1}. We note that \begin{equation*}
\begin{split}&\int_{\O}\Phi^{h,\varepsilon}\cdot\u_t\,dx=\frac{1}{2}\frac{d}{dt}\int_{\O}|(\theta_h\u)^{\varepsilon}|^2\,dx,
\end{split}
\end{equation*}
which yields
 \begin{equation}
\label{ut term}
\begin{split}&\lim_{(\varepsilon,h)\to0}\int_{0}^{t}\int_{\O}\Phi^{h,\varepsilon}\cdot\u_t\,dx=\frac{1}{2}\int_{\O}|\u(t)|^2\,dx-\frac{1}{2}\int_{\O}|\u_0|^2\,dx.
\end{split}
\end{equation}

\vskip0.3cm
\textbf{Step 2}. In this step, we have the following proposition on the convection term:
\begin{Proposition}
\label{prop. on convection}
Let $\u$ be as in Theorem \ref{main result}, then
\begin{equation*}
\begin{split}&
\int_0^t\int_{\O}\Dv(\u\otimes\u)\theta_h((\theta_h\u)^{\varepsilon})^{\varepsilon}\,dx\,dt\to 0
\end{split}
\end{equation*}
as $(\varepsilon,h)\to 0$.
\end{Proposition}
\begin{proof}
 consider the convection term
\begin{equation*}
\begin{split}&
\int_{\O}\Dv(\u\otimes\u)\theta_h((\theta_h\u)^{\varepsilon})^{\varepsilon}\,dx
\\&=-\int_{\O}\u\otimes\u\nabla\theta_h(x)((\theta_h\u)^{\varepsilon})^{\varepsilon}\,dx
-\int_{\O}\u\otimes\u\theta_h(x)\nabla((\theta_h\u)^{\varepsilon})^{\varepsilon}\,dx
\\&=A_1+A_2.
\end{split}
\end{equation*}
We see that
\begin{equation*}
\begin{split}
A_1&=-\int_{\O}\u\otimes\u\nabla\theta_h(x)((\theta_h\u)^{\varepsilon})^{\varepsilon}\,dx
\\&=\int_{\O}(\u\otimes\u\nabla\theta_h(x))^{\varepsilon}(\theta_h\u)^{\varepsilon}\,dx
\\&=\int_{\O}\left((\u\otimes\u\nabla\theta_h(x))^{\varepsilon}-(\u\otimes\u)^{\varepsilon}\nabla\theta_h\right)(\theta_h\u)^{\varepsilon}\,dx
\\&+\int_{\O}\left((\u\otimes\u)^{\varepsilon}\nabla\theta_h(x)-\u\otimes\u^{\varepsilon}\nabla\theta_h\right)(\theta_h\u)^{\varepsilon}\,dx
+\int_{\O}\u\otimes\u^{\varepsilon}\nabla\theta_h(\theta_h\u)^{\varepsilon}\,dx
\\&=A_{11\varepsilon}+A_{12\varepsilon}+A_{13\varepsilon}.
\end{split}
\end{equation*}
Note that $\u$ is bounded in $L^4(0,T;L^4(\O))$ and the property of $\theta_h$, one obtains that
\begin{equation*}
\begin{split}
|\int_0^tA_{11\varepsilon}\,dt|&\leq \int_0^T\int_{\O}\left|(\u\otimes\u\nabla\theta_h(x))^{\varepsilon}-(\u\otimes\u)^{\varepsilon}\nabla\theta_h\right||(\theta_h\u)^{\varepsilon}|\,dx\,dt
\\&\leq C\|\u\|^3_{L^4(0,T;L^4(\O))}\frac{\varepsilon}{h}\to 0,
\end{split}
\end{equation*}
where we used Lemma \ref{key lemma} and Remark \ref{remark for proof}, $h=\varepsilon^{\nu}$ for some $0<\nu<1$.
 Note that $\nabla\u$ is bounded in $L^2(0,T;L^2(\O))$, thus we have
  \begin{equation*}
\begin{split}
|\int_0^TA_{12\varepsilon}\,dt|&\leq \int_0^T\int_{\O}\left|(\u\otimes\u)^{\varepsilon}\nabla\theta_h(x)-\u\otimes\u^{\varepsilon}\nabla\theta_h\right||(\theta_h\u)^{\varepsilon}|\,dx\,dt
\\&\leq C\frac{\varepsilon}{h}\|\u\|^2_{L^4(0,T;L^4(\O))}\|\nabla\u\|_{L^2(0,T;L^2(\O))}\to 0
\end{split}
\end{equation*}
as $\varepsilon\to 0.$
Note that
\begin{equation*}
\begin{split}
A_{13\varepsilon}&=
\int_{\O}\u\otimes\u^{\varepsilon}\nabla\theta_h(\theta_h\u)^{\varepsilon}\,dx
\\&=
\int_{\O}\u\otimes(\u^{\varepsilon}(x)-\u^\varepsilon(\delta(x)))\nabla\theta_h(\theta_h\u)^{\varepsilon}\,dx,
\end{split}
\end{equation*}
thus
\begin{equation}
\begin{split}
\label{A13 control}
|\int_0^TA_{13\varepsilon}\,dt|&\leq\frac{C}{h}\||\u|^2\|_{L^4(0,T;L^4(\O))}\|\u^{\varepsilon}(x)-\u^{\varepsilon}(\delta(x))\|_{L^s(0,T;L^s(\O))}h^{\frac{s-2}{2s}}
\\&\leq C\||\u|^2\|_{L^4(0,T;L^4(\O))}\|\nabla\u^{\varepsilon}(x)\|_{L^s(0,T;L^s(\O))}h^{\frac{s-2}{2s}}
\\&\leq C\||\u|^2\|_{L^4(0,T;L^4(\O))}\|\u^{\varepsilon}\|_{L^s(0,T;{B_s^{\alpha,\infty}( \O)})}h^{\frac{s-2}{2s}}\varepsilon^{\alpha-1},
\end{split}
\end{equation}
where we used $$\|\u^{\varepsilon}(x)-\u^{\varepsilon}(\delta(x))\|_{L^s(0,T;L^s(\O))}\leq C(s,\O) h\|\nabla u^{\varepsilon}\|_{L^s(0,T;L^s(\O))}.$$
We choose $h=\varepsilon^\nu$, where \begin{equation}
\label{range of index}\frac{2s(1-\alpha)}{s-2}<\nu<1
\end{equation}  for any $s>2$, so $\alpha>\frac{1}{2}+\frac{1}{s}$, then $$
|\int_0^TA_{13\varepsilon}\,dt|\to 0$$ as $(\varepsilon,h)\to 0.$
\\

For term $A_2$, we calculate it as follows
\begin{equation*}
\begin{split}A_2&=
-\int_{\O}\u\otimes\u\theta_h(x)\nabla((\theta_h\u)^{\varepsilon})^{\varepsilon}\,dx
\\&=
-\int_{\O}(\u\otimes\theta_h(x)\u)^{\varepsilon}\nabla((\theta_h\u)^{\varepsilon})\,dx
\\&=
-\int_{\O}\left((\u\otimes\theta_h(x)\u)^{\varepsilon}-\u\otimes(\theta_h\u)^{\varepsilon}\right)\nabla((\theta_h\u)^{\varepsilon})\,dx+
\int_{\O}\u\otimes(\theta_h\u)^{\varepsilon}\cdot\nabla((\theta_h\u)^{\varepsilon})\,dx
\\&=A_{21\varepsilon}+A_{22\varepsilon}.
\end{split}
\end{equation*}
Note that
\begin{equation*}
\begin{split}&
-\int_{\O}\left((\u\otimes\theta_h(x)\u)^{\varepsilon}-\u\otimes(\theta_h\u)^{\varepsilon}\right)\nabla((\theta_h\u)^{\varepsilon})\,dx
\\&=
\int_{\O}\left(\Dv(\u\otimes\theta_h(x)\u)^{\varepsilon}-\Dv(\u\otimes(\theta_h\u)^{\varepsilon})\right)(\theta_h\u)^{\varepsilon}\,dx,
\end{split}
\end{equation*}
which could be bounded by $$C\|\nabla\u\|_{L^2(0,T;L^2(\O))}\|\theta_h\u\|_{L^4(0,T;L^4(\O))}.$$
 By Lemma \ref{Lions's lemma}, thus we have $A_{21\varepsilon}$ converges to zero as $\varepsilon\to 0$. This convergence does not depend on the value of $h$.
Meanwhile, we find that
$$\int_{\O}\u\otimes(\theta_h\u)^{\varepsilon}\cdot\nabla((\theta_h\u)^{\varepsilon})\,dx=\int_{\O}\u\cdot\nabla\frac{1}{2}|(\theta_h\u)^{\varepsilon}|^2\,dx=0,$$
because $\Dv\u=0$ in the distribution sense.
\end{proof}
\vskip0.3cm
\textbf{Step 3}. The goal of this step is to prove the following proposition.
\begin{Proposition}
\label{prop. viscous term}
Let $\u$ be as in Theorem \ref{main result}, then
$$\mu\int_0^t\int_{\O}\Delta\u\left(\theta_h((\theta_h\u)^{\varepsilon})^{\varepsilon}\right)\,dx\,dt\to\mu\int_0^t\int_{\O}|\nabla\u|^2\,dx\,dt$$
as $(\varepsilon,h)\to 0$.
\end{Proposition}
\begin{proof} For the viscous term, we have
\begin{equation}
\label{disserpation term}
\begin{split}&\mu\int_{\O}\Phi^{h,\varepsilon}\Delta\u\,dx=\mu\int_{\O}\Delta\u\left(\theta_h((\theta_h\u)^{\varepsilon})^{\varepsilon}\right)\,dx
\\&=-\mu\int_{\O}\nabla\u\cdot\nabla\theta_h((\theta_h\u)^{\varepsilon})^{\varepsilon}\,dx-\mu\int_{\O}(\theta_h\nabla\u)\cdot\nabla((\theta_h\u)^{\varepsilon})^{\varepsilon}\,dx
\\&=B_1+B_2.
\end{split}
\end{equation}
For term $B_1$,
we have
\begin{equation*}\begin{split}&
-\mu\int_{\O}\nabla\u\cdot\nabla\theta_h((\theta_h\u)^{\varepsilon})^{\varepsilon}\,dx=
\mu\int_{\O}(\nabla\u\cdot\nabla\theta_h)^{\varepsilon}((\theta_h\u)^{\varepsilon})\,dx
\\&=\mu\int_{\O}\left((\nabla\u\cdot\nabla\theta_h)^{\varepsilon}-\nabla\u^{\varepsilon}\cdot\nabla\theta_h\right)((\theta_h\u)^{\varepsilon})\,dx+
\mu\int_{\O}\nabla\u^{\varepsilon}\cdot\nabla\theta_h((\theta_h\u)^{\varepsilon}-\theta_h\u^{\varepsilon})\,dx
\\&+\mu\int_{\O}\nabla\u^{\varepsilon}(\theta_h\u^{\varepsilon})\cdot\nabla\theta_h\,dx
\\&=B_{11}+B_{12}+B_{13}.
\end{split}
\end{equation*}
We are able to control $B_{11}$ as follows
\begin{equation*}
\begin{split}&\int_0^TB_{11}\,dt
\\&\leq
\mu\left|\int_0^T\int_{\O}\left((\nabla\u\cdot\nabla\theta_h)^{\varepsilon}-\nabla\u^{\varepsilon}\cdot\nabla\theta_h\right)((\theta_h\u)^{\varepsilon})\,dx\,dt\right|
\\&\leq C \int_0^T\int_{\O_h}\left|(\nabla\u\cdot\nabla\theta_h)^{\varepsilon}-\nabla\u^{\varepsilon}\nabla\theta_h\right|\left|((\theta_h\u)^{\varepsilon})\right|\,dx\,dt
\\&\leq C \|\nabla\u\|_{L^2(0,T;L^2(\O))}\|\u\|_{L^4(0,T;L^4(\O))}\frac{\varepsilon}{h^{\frac{3}{4}}},
\end{split}
\end{equation*}
and control $B_{12}$ as follows
\begin{equation*}
\begin{split}&\left|
\mu\int_0^T\int_{\O}\nabla\u^{\varepsilon}\cdot\nabla\theta_h((\theta_h\u)^{\varepsilon}-\theta_h\u^{\varepsilon})\,dx\,dt\right|
\\&\leq C\frac{\varepsilon}{h^{\frac{3}{4}}}\|\nabla\u\|_{L^2(0,T;L^2(\O))}\|\u\|_{L^4(0,T;L^4(\O))},
\end{split}
\end{equation*}
choosing $h=\varepsilon^{\nu}$, where $\nu$ is given in \eqref{range of index}, then $B_{11}$ and $B_{12}$ tend to zero.
We give a control on $B_{13}$,
\begin{equation*}
\mu\int_0^T\int_{\O}\nabla\u^{\varepsilon}(\theta_h\u^{\varepsilon})\cdot\nabla\theta_h\,dx\,dt
=\mu\int_0^T\int_{\O_h}\nabla\u^{\varepsilon}\theta_h(\u^{\varepsilon}(x)-\u^\varepsilon(\delta(x)))\cdot\nabla\theta_h\,dx\,dt,
\end{equation*}
which could be bounded by
\begin{equation}
\begin{split}&\frac{C\mu}{h}\||\nabla\u|\|_{L^2(0,T;L^2(\O))}\|\u^{\varepsilon}(x)-\u^{\varepsilon}(\delta(x))\|_{L^s(0,T;L^s(\O))}h^{\frac{s-2}{2s}}
\\&\leq C\mu\|\nabla\u\|_{L^2(0,T;L^2(\O))}\|\nabla\u^{\varepsilon}(x)\|_{L^s(0,T;L^s(\O))}h^{\frac{s-2}{2s}}
\\&\leq C\mu\|\nabla\u\|_{L^2(0,T;L^2(\O))}\|\u^{\varepsilon}\|_{L^s(0,T;{B_s^{\alpha,\infty}( \O)})}h^{\frac{s-2}{2s}}\varepsilon^{\alpha-1}.
\end{split}
\end{equation}
Choosing $h=\varepsilon^{\nu}$, where $\nu$ is given in \eqref{range of index}, one obtains $B_1\to 0$
as $(\varepsilon,h)\to 0.$
For term $B_2$, we have
\begin{equation*}
\begin{split}
&\mu\int_{\O}(\theta_h\nabla\u)\cdot\nabla((\theta_h\u)^{\varepsilon})^{\varepsilon}\,dx=
\mu\int_{\O}(\theta_h\nabla\u)^{\varepsilon}\cdot\nabla(\theta_h\u)^{\varepsilon}\,dx
\\&=\mu\int_{\O}(\theta_h\nabla\u)^{\varepsilon}\cdot\nabla(\theta_h\u)^{\varepsilon}\,dx
-\mu\int_{\O}(\theta_h\nabla\u^{\varepsilon})\cdot\nabla((\theta_h\u)^{\varepsilon})\,dx
\\&+\mu\int_{\O}(\theta_h\nabla\u^{\varepsilon})\cdot\nabla((\theta_h\u)^{\varepsilon})\,dx-\mu\int_{\O}(\theta_h\nabla\u^{\varepsilon})\cdot\nabla(\theta_h\u^{\varepsilon})
\,dx
+
\mu\int_{\O}(\theta_h\nabla\u^{\varepsilon})\cdot\nabla(\theta_h\u^{\varepsilon})\,dx
\\&=B_{21,\varepsilon}+B_{22,\varepsilon}+
\mu\int_{\O}(\theta_h\nabla\u^{\varepsilon})\cdot\nabla(\theta_h\u^{\varepsilon})\,dx.
\end{split}
\end{equation*}
We find
\begin{equation*}
\begin{split}B_{21,\varepsilon}&=- \mu\int_{\O}\left(\Dv(\theta_h\nabla\u)^{\varepsilon}-\Dv(\theta_h\nabla\u^{\varepsilon})\right)(\theta_h\u)^{\varepsilon}\,dx
\\&\leq C\mu\|\left(\Dv(\theta_h\nabla\u)^{\varepsilon}-\Dv(\theta_h\nabla\u^{\varepsilon})\right)\|_{L^{\frac{4}{3}}(0,T;L^{\frac{4}{3}}(\O))}\|\u\|_{L^4(0,T;L^4(\O))},
\end{split}
\end{equation*}
and Lemma \ref{Lions's lemma} yields $$\|\left(\Dv(\theta_h\nabla\u)^{\varepsilon}-\Dv(\theta_h\nabla\u^{\varepsilon})\right)\|_{L^{\frac{4}{3}}(0,T;L^{\frac{4}{3}}(\O))}\leq C\frac{\varepsilon}{h}\|\nabla\u\|_{L^2(0,T;L^2(\O))}$$
choosing $h=\varepsilon^{\nu}$, where $\nu$ is given in \eqref{range of index}, one obtains that$B_{21\varepsilon}\to 0$
as $(\varepsilon,h)\to 0.$
Similarly, we have  $$B_{22,\varepsilon}\to 0$$
 as $(\varepsilon,h)\to 0$.

 The last term of $B_2$ is given by
 \begin{equation}
 \label{last term of B2}
 \begin{split}&
\mu\int_{\O}(\theta_h\nabla\u^{\varepsilon})\cdot\nabla(\theta_h\u^{\varepsilon})\,dx=
\mu\int_{\O}|\theta_h\nabla\u^{\varepsilon}|^2\,dx+
\mu\int_{\O}(\theta_h\nabla\u^{\varepsilon}):\u^{\varepsilon}\nabla\theta_h\,dx.
\end{split}
\end{equation}
We calculate the second term on the right side in \eqref{last term of B2}
\begin{equation*}\begin{split}
&|
\mu\int_0^T\int_{\O}(\theta_h\nabla\u^{\varepsilon}):\u^{\varepsilon}\nabla\theta_h\,dx\,dt|
\\&\leq \frac{C}{h}
\mu|\int_0^T\int_{\O}(\theta_h\nabla\u^{\varepsilon})(\u^{\varepsilon}(x)-\u^{\varepsilon}(\delta(x)))\,dx\,dt|
\\&\leq \frac{C\mu}{h}\||\nabla\u|\|_{L^2(0,T;L^2(\O))}\|\u^{\varepsilon}(x)-\u^{\varepsilon}(\delta(x))\|_{L^s(0,T;L^s(\O))}h^{\frac{s-2}{2s}}
\\&\leq C\mu\|\nabla\u\|_{L^2(0,T;L^2(\O))}\|\nabla\u^{\varepsilon}(x)\|_{L^s(0,T;L^s(\O))}h^{\frac{s-2}{2s}}
\\&\leq C\mu\|\nabla\u\|_{L^2(0,T;L^2(\O))}\|\u^{\varepsilon}\|_{L^s(0,T;{B_s^{\alpha,\infty}( \O)})}h^{\frac{s-2}{2s}}\varepsilon^{\alpha-1}\to 0
\end{split}
\end{equation*}
as $(\varepsilon,h)\to 0$. We find the first term on the right side in  \eqref{last term of B2} as follows
\begin{equation*}
\begin{split}
\mu\int_{\O}|\theta_h\nabla\u^{\varepsilon}|^2\,dx&=\mu\int_{d(x,\partial\O)\geq 2h}|\nabla\u^{\varepsilon}|^2\,dx+\mu\int_{h\leq d(x,\partial\O)\leq 2h}|\theta_h\nabla\u^{\varepsilon}|^2\,dx
\\&=J_{1\varepsilon}+J_{2\varepsilon}.
\end{split}
\end{equation*}
Note that $\nabla\u$ is bounded in $L^2(0,T;L^2(\O))$,   we have
$$J_{1\varepsilon}\to \mu\int_{\O}|\nabla\u|^2\,dx$$
as $(\varepsilon,h)\to 0$.
Term $J_{2\varepsilon}$ could be controlled as follow
\begin{equation*}
\begin{split}&
|J_{2\varepsilon}|\leq C\mu\int_{h\leq d(x,\partial\O)\leq 2h}|\nabla\u^{\varepsilon}|^2\,dx
\\&\leq C\mu\varepsilon^{\alpha-1}\|\u\|_{B_s^{\alpha,\infty}( \O)}h^{\frac{s-2}{s}}.
\end{split}
\end{equation*}
Choosing $h=\varepsilon^{\frac{s(1-\alpha)}{s-2}}$  for any $s>2$ and $0<\alpha<1,$ thus $$J_{2\varepsilon}\to 0$$ as $(\varepsilon,h)\to 0.$
\end{proof}
\vskip0.3cm
\textbf{Step 4}. We are aiming at proving the following proposition with respect to Pressure term.
\begin{Proposition}
\label{prop pressure}
Let $\u$ be as in Theorem \ref{main result}, then
\begin{equation*}
\int_0^t\int_{\O}\nabla P(\theta_h(\theta_h\u)^{\varepsilon})^{\varepsilon}\,dx\,dt\to 0
\end{equation*}
as $(\varepsilon,h)\to 0$.
\end{Proposition}
\begin{proof}
 $P$ is a solution to the following equation
\begin{equation*}
-\Delta P=\sum_{i,j=1}^{N}\partial_{x_i}\partial_{x_j}(\u_i\u_j),\quad P=0\text{ on }\partial\O.
\end{equation*}
This gives us \begin{equation}
\label{estimate P}
\|P\|_{L^{\frac{p}{2}}(0,T;L^{\frac{q}{2}}(\O))}\leq C\||\u|^2\|_{L^{\frac{p}{2}}(0,T;L^{\frac{q}{2}}(\O))}.
\end{equation}
Here we consider the following ones
\begin{equation*}
\begin{split}&
\int_{\O}\nabla P(\theta_h(\theta_h\u)^{\varepsilon})^{\varepsilon}\,dx
\\&=-\int_{\O}(P\nabla\theta_h(x))^{\varepsilon}(\theta_h\u)^{\varepsilon}\,dx
-\int_{\O}(P\theta_h(x))^{\varepsilon}\Dv(\theta_h\u)^{\varepsilon}\,dx
\\&=D_{1\varepsilon}+D_{2\varepsilon}.
\end{split}
\end{equation*}
Note the definition of function $\theta_h(x)$, we are able to show that $D_{1\varepsilon}\to 0$. In fact,
\begin{equation*}
\begin{split}&
 |\int_{\O}(P\nabla\theta_h(x))^{\varepsilon}(\theta_h\u)^{\varepsilon}\,dx|=|\int_{h\leq d\leq 2h}(P\nabla\theta_h(x))^{\varepsilon}(\theta_h\u)^{\varepsilon}\,dx|
\\&\leq C\||\u|^3\|_{L^{\frac{q}{3}}(\O)} h^{\frac{q-3}{q}},
\end{split}
\end{equation*}
this tends to zero because $q\geq 4.$
For term $D_{2\varepsilon}$, we have
\begin{equation*}
\begin{split}
D_{2\varepsilon}&=-\int_{\O}(P\theta_h(x))^{\varepsilon}\Dv(\theta_h\u)^{\varepsilon}\,dx+\int_{\O}(P\theta_h(x))^{\varepsilon}\Dv(\theta_h\u^{\varepsilon})\,dx
\\&-\int_{\O}(P\theta_h(x))^{\varepsilon}\Dv(\theta_h\u^{\varepsilon})\,dx
\\&=D_{21\varepsilon}+D_{22\varepsilon}.
\end{split}
\end{equation*}
By Lemma \ref{Lions's lemma}, for
choosing $h=\varepsilon^{\nu}$, where $\nu$ is given in \eqref{range of index}, we have $D_{21\varepsilon}\to 0$ as $\varepsilon$ goes to zero. Thus, we focus on term $D_{22\varepsilon}$ in the following ones
\begin{equation*}
\begin{split}
D_{22\varepsilon}&=
-\int_{\O}(P\theta_h(x))^{\varepsilon}\Dv(\theta_h\u^{\varepsilon})\,dx
\\&=-\int_{\O}(P\theta_h(x))^{\varepsilon}\nabla\theta_h\u^{\varepsilon}\,dx-\int_{\O}(P\theta_h(x))^{\varepsilon}\theta_h\Dv\u^{\varepsilon}\,dx
\\&=-\int_{\O}(P\theta_h(x))^{\varepsilon}\nabla\theta_h\u^{\varepsilon}\,dx,
\end{split}
\end{equation*}
using the same argument for $D_{1\varepsilon},$ we are able to show that $D_{22\varepsilon}$ tends to zero as $(\varepsilon,h)\to 0.$
\end{proof}
\vskip0.3cm
\textbf{Step 5}.
Taking integration on both sides of \eqref{first step equality} with respect to $t$, and by \eqref{ut term}, Proposition \ref{prop. on convection}, Proposition \ref{prop. viscous term} and Proposition \ref{prop pressure},  one obtains the following equality
\begin{equation*}
\frac{1}{2}\int_{\O}|\u|^2\,dx+\mu\int_0^t\int_{\O}|\nabla\u|^2\,dx\,dt=\frac{1}{2}\int_{\O}|\u_0|^2\,dx.
\end{equation*}
\section*{Acknowledgements}

C. Yu's research was supported in part by Professor
Caffarelli's NSF Grant  DMS-1540162.


\end{document}